\documentclass{amsart}
\usepackage{graphicx,amsmath,amssymb,amsthm} 
\usepackage{comment}
\usepackage[a4paper, tmargin=2in, bmargin=2in, lmargin=1.3in,rmargin=1.3in]{geometry}
\usepackage{hyperref}

\theoremstyle{plain}    
 \newtheorem{theorem}{Theorem}[section]
 \numberwithin{equation}{section} 
 \numberwithin{figure}{section} 
 \theoremstyle{plain}
 \theoremstyle{plain}    
 \newtheorem{corollary}[theorem]{Corollary} 
 \theoremstyle{plain}    
 \newtheorem{proposition}[theorem]{Proposition} 
 \theoremstyle{plain}    
 \newtheorem{lemma}[theorem]{Lemma} 
 \theoremstyle{remark}

 \theoremstyle{definition}

\theoremstyle{definition}
\newtheorem{definition}[theorem]{Definition}

\def \ddbar {i\partial\bar{\partial}}

\title{On K\"ahler-Einstein Currents}

\author[Y. Chen]{Yifan Chen}
\address{Department of Mathemetics, University of California, Berkeley, Berkeley CA, USA}
\email{yifan-chen@berkeley.edu}

\author[S.-K. Chiu]{Shih-Kai Chiu}
\address{Department of Mathematics, University of California, Irvine, Irvine CA, USA}
\email{shihkaic@uci.edu}

\author[M. Hallgren]{Max Hallgren}
\address{Department of Mathematics, Rutgers University, Hill Center for the Mathematical Sciences 
110 Frelinghuysen Rd.
Piscataway, NJ, USA}
\email{mh1564@scarletmail.rutgers.edu}

\author[G. Sz\'ekelyhidi]{G\'abor Sz\'ekelyhidi}
\address{Department of Mathematics, Northwestern University, Evanston,
  IL, USA}
\email{gaborsz@northwestern.edu}

\author[T. D. T\^o]{Tat Dat T\^o}
\address{Institut de Math\'ematiques de Jussieu-Paris Rive Gauche, Sorbonne Universit\'e, 4 place Jussieu, 75252 Paris Cedex 05, France}
\email{tat-dat.to@imj-prg.fr}

\author[F. Tong]{Freid Tong}
\address{Department of Mathematics, University of Toronto, 40 St. George Street, Toronto, ON, Canada}
\email{freid.tong@utoronto.ca}

\usepackage{xcolor}
\date{}

\begin{document}

\begin{abstract}
  We show that a general class of singular K\"ahler metrics with Ricci curvature bounded below define K\"ahler currents. In particular the result applies to singular K\"ahler-Einstein metrics on klt pairs, and an analogous result holds for K\"ahler-Ricci solitons. In addition we show that if a singular K\"ahler-Einstein metric can be approximated by smooth metrics on a resolution whose Ricci curvature has negative part that is bounded uniformly in $L^p$ for $p > \frac{2n-1}{n}$, then the metric defines an RCD space. 
\end{abstract}

\maketitle

\section{Introduction}
Singular K\"ahler-Einstein metrics have been studied extensively recently. Yau's existence results~\cite{Yau78} in the case of negative or zero Ricci curvature have been extended to singular varieties by Eyssidieux-Guedj-Zeriahi~\cite{EGZ09}, while Chen-Donaldson-Sun's solution of the Yau-Tian-Donaldson conjecture~\cite{CDS3} on the existence of K\"ahler-Einstein metrics on Fano manifolds has been extended to the singular setting by Liu-Xu-Zhuang~\cite{LXZ22}, Li-Tian-Wang~\cite{LTW} and Li~\cite{Li22}. In fact the most natural setting is to consider klt pairs $(X, D)$, where $X$ is a normal projective variety, $D$ is a $\mathbb{Q}$-divisor, and $K_X + D$ is a $\mathbb{Q}$-line bundle that is either  negative, positive, or has a multiple that is trivial. In this case a K\"ahler-Einstein metric $\omega$ on the pair $(X,D)$ can be thought of as a K\"ahler-Einstein metric on $X^{\text{reg}}\setminus D$, with cone singularities along the smooth part of the divisor $D$, whose angles are determined by the coefficients. See Section~\ref{sec:setup} for the detailed definition. 

In this paper we study the metric geometry of such a singular K\"ahler-Einstein metric $\omega$. Our first result is the following.

\begin{theorem}\label{thm:Kcurrent}
  A singular K\"ahler-Einstein metric $\omega$ on a klt pair $(X,D)$ defines a K\"ahler current. In other words if $\omega_X$ is a smooth K\"ahler metric on $X$ (see Definition~\ref{defn:smoothKahler}), then we have $\omega > \epsilon \omega_X$ for some $\epsilon > 0$. 
\end{theorem}

The result follows from the following much more general statement for singular K\"ahler metrics with Ricci curvature bounded below.
\begin{theorem}\label{thm:generalKcurrent}
  Suppose that $(X, \omega_X)$ is a compact normal K\"ahler space, with smooth metric $\omega_X$. Suppose that $\omega = \omega_X + \ddbar u$ is a positive current with bounded potential satisfying the following properties:
  \begin{enumerate}
  \item $u$ is smooth on $X^{\text{reg}}\setminus D$ for a divisor $D$,
  \item $\omega^n = e^F \omega_X^n$, where $e^F\in L^p(\omega_X)$ for some $p > 1$, and $F\in L^1(\omega_X)$. 
  \item $\mathrm{Ric}(\omega) \geq -A(\omega+\omega_X)$ on $X^{\text{reg}}$ in the sense of currents, for some $A > 0$. 
  \end{enumerate}
  Then there is a constant $C > 0$, depending on $A, p$, an upper bound for $\Vert e^F\Vert_{L^p}$, and $(X,\omega_X)$, such that $\omega > C^{-1}\omega_X$. 
\end{theorem}

We will also give an extension of Theorem~\ref{thm:Kcurrent} to shrinking K\"ahler-Ricci solitons. 

\begin{theorem}\label{thm:KRScurrent}
  A singular shrinking K\"ahler-Ricci soliton $(X, \omega, V)$ on a klt pair $(X,D)$ defines a K\"ahler current. 
\end{theorem}

The question of whether singular K\"ahler-Einstein metrics define K\"ahler currents was considered previously by Guedj-Guenancia-Zeriahi~\cite{GGZ}. They prove the K\"ahler current property for smoothable varieties, and in three dimensions in the case of negative and zero Ricci curvature. The methods used in \cite{GGZ} rely essentially on approximations of the singular metric with smooth metrics that have Ricci curvature bounded below. See also \cite[p. 40]{PTT} and \cite[Proposition 4.3]{GP} for related results.

Our approach is based on the recent work \cite{Sz24} where approximations with more general metrics are combined with the heat kernel estimates of Guo-Phong-Song-Sturm~\cite{GPSS2}. In fact \cite[Theorem 16]{Sz24} shows the K\"ahler current property for singular K\"ahler-Einstein metrics under the assumption that the metric can be approximated with constant scalar curvature K\"ahler metrics. The new ingredient in this paper is the observation that it is enough to ensure that the approximating metrics have Ricci curvature controlled suitably in $L^1$, and that such approximations can always be found. To state this more precisely, let us introduce the following. 

\begin{definition}\label{defn:tameapprox}
  Let $\omega$ be a closed positive $(1,1)$-current on $X$ such that $\omega$ is smooth on $X^{\text{reg}}\setminus D$, and locally $\omega = i\partial\bar\partial u$ for bounded $u$. Let $\pi: Y \to X$ be a resolution of singularities that is a biholomorphism over $X^{\text{reg}}\setminus D$. We say that the smooth K\"ahler metrics $\omega_\epsilon$ on $Y$ are a \emph{tame approximation} of $\omega$, if we have the following: 
  \begin{itemize}
      \item[(a)] For smooth reference K\"ahler metrics $\omega_X$, $\omega_Y$ on $X,Y$, we have $\omega_\epsilon = \pi^*\omega_X + \epsilon \omega_Y + i\partial\bar\partial u_\epsilon$, and we have $\omega_\epsilon \to \pi^*\omega$ as $\epsilon \to 0$, locally smoothly on $\pi^{-1}(X^{\text{reg}}\setminus D)$, 
      \item[(b)] We have the estimates
      \[ \sup_Y |u_\epsilon| < C, \qquad \int_Y \left(\frac{\omega_\epsilon^n}{\omega_Y^n}\right)^p\, \omega_Y^n < C, \]
      for $C > 0$ and $p > 1$, independent of $\epsilon$. 
  \end{itemize}
\end{definition}

The purpose of these assumptions is that the results of Guo-Phong-Song-Sturm~\cite{GPSS2} can be applied uniformly to the family of metrics $\omega_\epsilon$. In Theorem~\ref{thm:generalKcurrent} we use a tame approximation of $\omega$ with metrics satisfying $\Vert (\mathrm{Ric}(\omega_\epsilon) + 2A( \omega_{\epsilon}+\omega_X))_-\Vert_{L^1(\omega_\epsilon)} \to 0$. 

To state our second main result, we define $\hat{X}$ to be the metric completion of $(X^{\text{reg}}\setminus D, \omega)$, where $\omega$ is a singular K\"ahler-Einstein metric on the pair $(X,D)$. Equipped with the volume form $\omega^n$, $\hat{X}$ defines a metric measure space. Recall that in \cite[Theorem 4]{Sz24} it was shown that if a singular K\"ahler-Einstein metric $\omega$ on $X$ has a tame approximation with constant scalar curvature metrics $\omega_\epsilon$, then the metric measure space $\hat{X}$ is a non-collapsed RCD space in the sense of de Philippis-Gigli~\cite{DPG}. In other words $\hat{X}$ satisfies synthetic Ricci curvature lower bounds. Our second main result is a strengthening of this result, requiring only that we have a tame approximation with metrics $\omega_\epsilon$, such that the negative parts $\mathrm{Ric}_-$ of their Ricci curvatures are uniformly bounded in $L^p$ for some $p > \frac{2n-1}{n}$, where $n$ is the complex dimension. 

\begin{theorem}\label{thm:LpRCD}
  Suppose that $\omega$ is a singular K\"ahler-Einstein metric on the klt pair $(X,D)$, and that $\omega$ has a tame approximation with metrics $\omega_\epsilon$ satisfying $\Vert \mathrm{Ric}_-\Vert_{L^p(\omega_\epsilon)} < C$, for $C$ independent of $\epsilon$, and $p > \frac{2n-1}{n}$. Then the metric measure space $\hat{X}$ is a non-collapsed RCD space. 
\end{theorem}

\noindent In fact a more general result holds, only requiring a lower bound for the Ricci curvature on $X^{\text{reg}}\setminus D$, see Theorem~\ref{thm:RCD2} for the precise statement. The proof of this result follows similar lines to the proof of \cite[Proposition 15]{Sz24} and the sharper result here follows from using an improved Kato inequality for the gradient of eigenfunctions. An application of the case $p=2$ is to the setting where a singular K\"ahler-Einstein metric on $X$ has tame approximations with extremal K\"ahler metrics, given in Corollary~\ref{cor:extremal}. This generalizes the result in \cite{Sz24}, where tame approximations with constant scalar curvature K\"ahler metrics were considered. Recently a similar result to Theorem~\ref{thm:LpRCD} was also obtained by Guo-Song~\cite{GS25}. 

This paper is organized as follows. In section~\ref{sec:setup}, we discuss some background and setup of the problem. Theorem~\ref{thm:generalKcurrent} will be proved in section~\ref{sec:KE}, and Theorem~\ref{thm:KRScurrent} is proved in section~\ref{sec:KRS}. Finally, we will prove Theorem~\ref{thm:LpRCD} in section~\ref{sec:RCD}. 

\subsection*{Acknowledgements}
This work arose from the AIM workshop "PDE methods in complex geometry", with further work performed at SLMath (formerly MSRI). We are grateful for the great working environment provided by them. We would like to thank Jacob Sturm for helpful discussions and also Bin Guo and Jian Song for sharing a preliminary version of their paper \cite{GS25}.  We would also like to thank Henri Guenancia, Valentino Tosatti, Chung-Ming Pan, and Duong H. Phong for helpful comments.

The work was supported in part by the following grants: 
S.C. and F.T. by NSF grant DMS-1928930; M.H. by NSF grant DMS-2202980;
G. Sz. by NSF grant DMS-2203218; T.D.T by ANR-21-CE40-0011-01,
PEPS-JCJC-2024 and Tremplins-2024.

\section{Background}\label{sec:setup}

Let $X$ be a compact normal complex analytic space. We recall the following notion of a smooth K\"ahler metric on $X$ from \cite{EGZ09,BBEGZ19}.
\begin{definition} \label{defn:smoothKahler}
A smooth K\"ahler metric on $X$ is defined by a K\"ahler metric $\omega$ on $X^{\text{reg}}$ such that for any $x\in X$, there is an open subset $U\ni x$ and a local holomorphic embedding $j_x: U\hookrightarrow \mathbb{C}^{N_x}$ such that $\omega|_{U\cap X^{\text{reg}}} = j_x^\ast \theta$ for some smooth K\"ahler metric $\theta$ on $\mathbb{C}^{N_x}$. The space $X$ is defined to be a singular K\"ahler space if it admits a smooth K\"ahler metric. 
\end{definition}

Let $(X,D)$ be a pair of a compact normal K\"ahler space $X$, equipped with an effective $\mathbb{Q}$-divisor $D$. Suppose that $K_X + D$ is a $\mathbb{Q}$-Cartier divisor. We say that the pair $(X,D)$ has klt singularities if the coefficients of $D$ are in $(0, 1)$ and there is a log resolution $\pi: Y\to X$ satisfying 
\[ K_Y +D'= \pi^{\ast}(K_X+D) + \sum_{i=1}^k a_i E_i, \]
where $a_i >-1$, $E_i$ are the exceptional divisors, and $D'$ is the strict transform of $D$. This condition is independent of the resolution chosen.

Following Berman-Boucksom-Eyssidieux-Guedj-Zeriahi~\cite{BBEGZ19}, we 
define singular K\"ahler-Einstein metrics on a klt pair $(X,D)$ as follows. Let $r$ be a positive integer such that $r(K_X+D)$ is Cartier. Let $\phi$ define a smooth Hermitian metric on the $\mathbb{Q}$-line bundle $K_X+D$. This means that if $\sigma$ is a nowhere vanishing holomorphic section of $r(K_X + D)$ over an open set $U$ for some $r > 0$, then $|\sigma|_{r\phi}$ is a smooth function on $U$ (a restriction of a smooth function under a local embedding into $\mathbb{C}^N$). We  then define the adapted measure  \[ \mu_\phi = \frac{ \Big(i^{rn^2} \sigma\wedge \bar\sigma\Big)^{1/r}}{ |\sigma|_{r\phi}^{2/r}}. \]
Note that in this expression we are locally viewing $\sigma^{1/r}$ as a holomorphic $n$-form on $X$ with poles along $D$. The Ricci curvature of $\mu_\phi$ is given by $\mathrm{Ric}(\mu_\phi) = -\Theta + [D]$, where $\Theta$ is the curvature of the metric $\phi$. 
A singular K\"ahler-Einstein metric on $(X,D)$ is defined as follows, depending on the positivity of the $\mathbb{Q}$-line bundle $K_X+D$:
\begin{itemize}
    \item If $K_X+D$ is ample, then we can choose the metric $\phi$ on $K_X+D$ above so that its curvature form is a smooth K\"ahler metric $\omega_X$. Then $\omega = \omega_X + i\partial\bar\partial u$ is a singular K\"ahler-Einstein metric if $u\in L^\infty$ satisfies the Monge-Amp\`ere equation
    \[ (\omega_X + i\partial\bar\partial u)^n = e^u \mu_\phi.\] 
    \item If $-(K_X+D)$ is ample, then we choose the metric $\phi$ above so that its curvature form is $\Theta=-\omega_X$ for a smooth K\"ahler metric $\omega_X$. In this case $\omega = \omega_X + i\partial\bar\partial u$ is a singular K\"ahler-Einstein metric if $u\in L^\infty$ satisfies
    \[ (\omega_X + i\partial\bar\partial u)^n =  e^{-u} \mu_\phi.\]
     
    \item Finally, if $r(K_X+D)$ is trivial for some $r > 0$, then we choose $\phi$ so that its curvature vanishes. Then, given a smooth K\"ahler metric $\omega_X$, the metric 
    $\omega = \omega_X + i\partial\bar\partial u$ is a singular K\"ahler-Einstein metric if $u\in L^\infty$ satisfies
    \[ (\omega_X + i\partial\bar\partial u)^n = c \mu_\phi, \]
    for a normalizing constant $c$. 
\end{itemize}
In each case, the Ricci curvature of the singular K\"ahler-Einstein metric can be defined as a current, and we have
\[ \mathrm{Ric}(\omega) = \lambda\omega + [D], \]
where $\lambda= -1, 1, 0$ in the three cases above, and $D$ is the current of integration along $D$. It is known (see Cho-Choi~\cite{CC24} and Coman-Guedj-Zeriahi~\cite{CGZ}) that in each case the potential $u$ is continuous on $X$, and smooth on $X^{\text{reg}}\setminus D$. Since the klt condition implies that the measure $\mu_\phi$ has $L^p$-density relative to $\omega_X^n$ for some $p > 1$, Theorem~\ref{thm:Kcurrent} follows immediately from Theorem~\ref{thm:generalKcurrent}. 

Let us now recall the results of Guo-Phong-Song-Sturm~\cite{GPSS2} that we need, for tame approximations of $\omega$ as in Definition~\ref{defn:tameapprox}. From \cite[Theorem 2.1, 2.2 and Corollary 2.1]{GPSS2} (see also \cite{GPSS3, GT24} on removing the assumption of the lower bound of the volume form) we have the following.
\begin{theorem}[See \cite{GPSS2}] \label{thm:GPSS}
  If the metrics $\omega_\epsilon$ are a tame approximation of $\omega$ on a resolution $Y\to X$, then we have the following: 
  \begin{itemize}
      \item[(a)] We have a uniform Sobolev inequality
      \[ \left(\int_Y |f - \overline{f}|^{2q}\, \omega_\epsilon^n\right)^{1/q} \leq C \int_Y |\nabla f|^2_{\omega_\epsilon}\, \omega_\epsilon^n, \]
      for some $q > 1$ (in fact for any $q < \frac{n}{n-1}$), where $\overline{f}:= \frac{1}{\int_Y \omega_{\epsilon}^n}\int_Y f\omega_{\epsilon}^n$. 
      \item[(b)] We have a uniform bound for the heat kernel $H(x,y,t)$ of $(Y, \omega_\epsilon)$ of the form
      \[ H(x,y,t) \leq C(t), \]
      for $t\in (0,2]$, where $C(t)$ is continuous and $C(t)\to\infty$ as $t\to 0$. 
      \item[(c)] For any $q < \frac{n}{n-1}$ we have a lower bound $\lambda_k \geq c k^{\frac{q-1}{q}}$ for the $k^{th}$ eigenvalue of the Laplacian on $(Y, \omega_\epsilon)$. Here $c$ depends on $q, \omega$, but not on $\epsilon$. 
  \end{itemize}
\end{theorem}

Finally, we briefly recall some background on RCD spaces. See Ambrosio-Gigli-Savar\'e~\cite{AGS} and de Philippis-Gigli~\cite{DPG} for detailed definitions. A metric measure space is a triple $(Z, d, \mu)$, where $(Z,d)$ is a metric space, and $\mu$ is a measure on $Z$. We will only deal with \emph{almost smooth} metric measure spaces in the sense of Honda~\cite{Honda} (see also \cite[Definition 5]{Sz24} for the slightly different definition we use). This means that $Z$ has a dense open subset $U$, such that the metric $d$ on $U$ is locally induced by a smooth Riemannian metric $g$, and we have the following conditions:
\begin{itemize}
    \item The restriction of the measure $\mu$ to $U$ coincides with the Riemannian volume measure of $(U, g)$.
    \item The set $Z\setminus U$ has zero capacity in the sense that there is an exhaustion $K_1\subset K_2\subset\ldots$ of $U$ with compact sets, and function $\phi_i$ supported in $U$, equal to 1 on $K_i$, with $\int_U |\nabla \phi_i|^2\, d\mu \to 0$ as $i\to\infty$. 
\end{itemize}

In our setting of a singular K\"ahler metric $\omega$ on the pair $(X,D)$, we define the metric space $\hat{X}$ to be the metric completion of the length metric space $(X^{\text{reg}}\setminus D, \omega)$. We equip $\hat{X}$ with the volume form $\omega^n$, pushed forward from $X^{\text{reg}}\setminus D$. Note that we have $\mathrm{Ric}(\omega)\geq \lambda\omega$, for a constant $\lambda$.
From Honda's work~\cite[Corollary 3.10]{Honda} (see also \cite[Corollary 8]{Sz24}) together with \cite[Lemma 3.7]{Song14} and the Sobolev inequality in Theorem~\ref{thm:GPSS}, we have the following (see \cite[Proposition 9]{Sz24}). 
\begin{proposition}\label{prop:Honda}
  The metric measure space $\hat{X}$ is a non-collapsed $RCD(2n, \lambda)$ space if and only if the eigenfunctions of the Laplacian on $X^{\text{reg}}\setminus D$ are Lipschitz. I.e. if $f\in W^{1,2}$ satisfies $\Delta f = -\mu f$ on $X^{\text{reg}}\setminus D$ for a constant $\mu$, then $|\nabla f|\in L^\infty(X^{\text{reg}}\setminus D)$. 
\end{proposition}

\section{Tame approximation with $L^1$ control of Ricci}\label{sec:KE}
In this section we prove Theorem~\ref{thm:generalKcurrent}. The method is similar to that in \cite[Theorem 16]{Sz24}, except we use a tame approximation of the singular K\"ahler metric $\omega$ with metrics $\omega_\epsilon$ on a resolution whose Ricci curvature is controlled in $L^1$ rather than with constant scalar curvature metrics. More precisely we have the following. 

\begin{proposition}\label{prop:L1approx}
    Suppose that $\omega = \omega_X + \ddbar u$ is a singular K\"ahler  metric on the normal K\"ahler space $(X,\omega_X)$, satisfying the conditions (1), (2), (3) in Theorem~\ref{thm:generalKcurrent}. We can then find a tame approximation of $\omega$ with metrics $\omega_\epsilon$ on a resolution $Y$, such that 
    \[ \Vert (\mathrm{Ric}(\omega_\epsilon) + 2A(\omega_{\epsilon}+\pi^*\omega_X))_-\Vert_{L^1(\omega_\epsilon)} \to 0\]
    as $\epsilon \to 0$. 
\end{proposition}

 The approximations that we use are analogous to the ones used by Eyssidieux-Guedj-Zeriahi~\cite{EGZ09} in constructing singular K\"ahler-Einstein metrics. We remark that Guenancia~\cite[Proposition 3.9]{Gue16}
 also implicitly used $L^1$ bounds for the Ricci curvature of such approximations (see also Druel-Guenancia-Paun~\cite{DGP} and Greb-Guenancia-Kebekus~\cite{GGK}). 

 \begin{proof}
 	Let $\pi : Y \to X$ be a resolution of $X$, admitting a K\"ahler
  metric $\omega_Y$. Such a resolution exists by \cite[Lemma 2.2]{CMM}.  We will construct a family of metrics
  \[ \omega_\epsilon = \pi^*\omega_X + \epsilon \omega_Y + \ddbar
  u_\epsilon \]
  on $Y$, that approximate $\omega$ as $\epsilon\to 0$, and such that
  we have good control of the negative part of
  $\mathrm{Ric}(\omega_\epsilon) + 2A(\omega_\epsilon+\pi^*\omega_X)$
  in $L^1$.  To simplify notation, we will often write $\omega_X$ for $\pi^*\omega_X$ below on $Y\setminus E$, where $E$ is the exceptional divisor.  
  
  The equation for $\omega = \omega_X + \ddbar u$ on $X$ is
  \[ (\omega_X + \ddbar u)^n = e^F \omega_X^n. \]
  The Ricci curvature bound shows that on $X^{\text{reg}}$ we have
  \begin{equation} \label{eq:41} \ddbar(-F) + \mathrm{Ric}(\omega_X) \geq - A(\omega +
  	\omega_X).
  \end{equation}
  Since $\omega_X$ is locally the restriction of a smooth metric under
  an embedding into $\mathbb{C}^N$, it follows that for a constant $C$
  we have $\mathrm{Ric}(\omega_X) \leq C\omega_X$ on $X^{\text{reg}}$. At the
  same time $\omega$ and $\omega_X$ has locally bounded potential. It
  follows that any point $x\in X$ has a neighborhood $U$ and a bounded
  function $v$ on $U$ such that on
  $U\cap X^{\text{reg}}$ we have 
  \[ \ddbar(v-F) \geq 0. \]
  It follows from this that $v-F$ has a psh extension to $U$, and so it is bounded above in a neighborhood of $x$ (see Grauert-Remmert~\cite{GR56}). Therefore $-F$ is bounded above.
  
  Also, on $Y$ the form $\pi^*\omega_X$
  is locally given by pulling back a smooth metric along
  a holomorphic map to $\mathbb{C}^N$. It follows that away from the
  exceptional set on $Y$ we have $\pi^*\omega_X^n = e^g \omega_Y^n$,
  where
  \begin{equation} \label{eq:gdefn}
  g = \sum_i a_i \log |s_i|^2_{h_i}, 
  \end{equation}
  in terms of defining sections $s_i$ of the line bundles $\mathcal{O}(E_i)$
  corresponding to the exceptional divisors, and smooth Hermitian
  metrics $h_i$. Note that all $a_i > 0$ here. It follows that on
  $Y\setminus E$ we have
  \[ \mathrm{Ric}(\omega_X) = -\ddbar g + \mathrm{Ric}(\omega_Y) =
  \mathrm{Ric}(\omega_Y) + \sum_i a_i \Theta_{h_i}. \]
   From \eqref{eq:41} we then have
  \begin{equation}\label{eq:ddbG10} \ddbar(Au-F) + \mathrm{Ric}(\omega_Y) + \sum_i a_i\Theta_{h_i}
  \geq -2A\omega_X \end{equation}
  on $Y\setminus E$, and the upper bound on $-F$ implies that this inequality extends across the exceptional set $E$.

  Now we construct the tame approximation $\omega_{\epsilon} = \pi^{\star}\omega_X+\epsilon\omega_Y+\ddbar u_{\epsilon}$ by solving the equation
  \begin{equation}\label{eq: tame-approx-eqn}
  	(\pi^{\star}\omega_X+\epsilon\omega_Y+\ddbar u_{\epsilon})^n = e^{Au_{\epsilon}-G_{\epsilon}+g_{\epsilon}}\omega_Y^n
  \end{equation}
	where $G_{\epsilon}$ is a regularization of $Au-F$ and $g_{\epsilon}$ is a regularization of $g$ to be defined as follows. We write $G = Au-F$, and let $G_\epsilon$ be a smoothing of $\max(G, -\epsilon^{-1})$ using \cite{Blocki-Kol}. Note that we have
	\[ \begin{aligned}
      \ddbar (G) + \mathrm{Ric}(\omega_Y) + \sum_i a_i
      \Theta_{h_i} &\geq -C_1
      \omega_Y, \\
       \ddbar (-\epsilon^{-1}) + \mathrm{Ric}(\omega_Y) + \sum_i a_i
      \Theta_{h_i} &\geq -C_1 \omega_Y, \end{aligned} \]
    so it follows that
	\[
    \ddbar \max\{G, -\epsilon^{-1}\} +  \mathrm{Ric}(\omega_Y) + \sum_i a_i
      \Theta_{h_i} \geq - C_1\omega_Y.
    \]
    We can therefore construct the smoothing $G_\epsilon$ in such a way that is satisfies
  \begin{equation}\label{eq:Gpsh} \ddbar G_\epsilon +  \mathrm{Ric}(\omega_Y) + \sum_i a_i
      \Theta_{h_i} \geq - 2C_1\omega_Y,
    \end{equation}    
    $G_\epsilon\searrow  G$ locally smoothly on $\pi^{-1}(X^{reg}\setminus D)$, and 
 $G_\epsilon \geq \max\{G, -\epsilon^{-1}\}$.

    Similarly, we define $g_\epsilon$ to be a smoothing of $\max\{g, -\epsilon^{-1}\}$, satisfying the properties
    \begin{itemize}
        \item $\displaystyle{ \ddbar g_\epsilon \geq -C_1\omega_Y}$, 
        \item $g_\epsilon \searrow g$ locally smoothly on $\pi^{-1}(X^{reg}\setminus D)$, 
        \item $g_\epsilon \leq \max\{g, -\epsilon^{-1}\} +1$. 
    \end{itemize}

  Now we show this is a tame approximation.
 First we check the uniform $L^p$ bound for $e^{-G_\epsilon+g_\epsilon}$. To see this, note that
  by the boundedness of $u$, we have
  \[\int_Y e^{p(-G)}\, \omega_Y^n = \int_Y e^{p(F-Au)}\, \omega_Y^n\leq C\int_Ye^{pF}\,\omega_Y^n <\infty \]
  By the properties of $G_\epsilon, g_\epsilon$ above, we have $-G_{\epsilon}\leq -G$ and $g_{\epsilon}\leq C$, it follows that on $Y$ we have 
  \[ \int_Y e^{p(-G_\epsilon+g_\epsilon)}\,\omega_Y^n\leq C\int_Y e^{-pG}\omega_Y^n < C, \]
  for a constant $C$ independent of $\epsilon$. 
  It follows in particular from \cite[Theorem 4.1]{EGZ09} that we have a
  uniform $L^\infty$ bound $|u_\epsilon| < C$.

Since the $u_\epsilon$ are $\omega_X + \omega_Y$-psh, and uniformly bounded, so up
to choosing a subsequence we have a limit $u_\infty \in L^\infty(Y)$ such that
$e^{u_\epsilon} \to e^{u_\infty}$ in $L^q$ for all $q$. We now show that $u_\infty$ is continuous on $Y\setminus E$ and solves the Monge-Amp\`ere equation 
\[ (\pi^*\omega_X + \ddbar u_\infty)^n = e^{Au_\infty -G +
    g}\omega_Y^n = e^{Au_\infty + F - Au}\omega_X^n. \]
It suffices to show that $u_\epsilon$ converges locally uniformly in $Y\setminus E$ to $u_\infty$.  

We shall follow the argument in the proof of \cite[Theorem A]{GL}. First for any $\delta > \epsilon$, we have 
   \[
(\pi^*\omega_X +\delta\omega_Y + \ddbar u_{\epsilon})^n \geq  (\pi^* \omega_X +\epsilon\omega_Y + \ddbar u_{\epsilon})^n  =  e^{Au_\epsilon - G_\epsilon+g_\epsilon } \omega_Y^n.    
   \]
 By the stability estimates in \cite[Theorem 3.3]{GL}, we get 
\begin{equation}\label{eq_est_1}
u_{\epsilon}\leq u_\delta+ C\Vert  e^{-G_\delta+g_\delta} - e^{-G_\epsilon+g_\epsilon}  \Vert_{L^p(\omega_Y^n)}^{1/n}.
\end{equation} 
Since  $ \omega_Y = \pi^{\ast} \omega_X - \sum_{i=1}^k b_i \Theta_{h_i}$
for some $b_i >0$ and  smooth hermitian metrics $h_i$ (cf. \cite[Lemma 2.2]{CMM}), we can take $\rho= \sum_i b_i \log |s_i|^2_{h_i}-C'$  such that $  \pi^*\omega_X+ \ddbar \rho \geq \omega_Y $,  $\sup_Y \rho=0$, and $\{\rho=-\infty\} = E$. We consider $v_\delta:= (1-\delta )u_\delta + \delta \rho - C_1 \delta$ with $\|u_\delta \|_{L^\infty} +nA^{-1}\leq  C_1$. Then  we have
 \begin{align*}
 \pi^* \omega_X &+ \epsilon\omega_Y+ \ddbar v_\delta = \pi^* \omega_X + \epsilon\omega_Y + \delta \ddbar  \rho + (1-\delta) \ddbar u_\delta\\
&  \geq \delta ( \pi^*  \omega_X + \ddbar \rho  ) + (1-\delta)( \pi^* \omega_X + \delta\omega_Y + \ddbar u_\delta) - \delta\omega_Y\\
&\geq \delta\omega_Y + (1-\delta)( \pi^* \omega_X + \delta\omega_Y + \ddbar u_\delta) -\delta\omega _Y \\
&=   (1-\delta)( \pi^* \omega_X + \delta\omega_Y + \ddbar u_\delta) 
 \end{align*}
hence
\[
\begin{aligned} 
(\omega_X + \epsilon\omega_Y+ \ddbar v_\delta)^n &\geq  (1-\delta)^n (\pi^* \omega_X + \delta\omega_Y +\ddbar u_\delta)^n \\
&= (1-\delta)^n e^{A u_{\delta} -G_\delta+g_\delta }\omega_Y^n \\
&= (1-\delta)^n e^{A (1-\delta)u_{\delta} + A\delta u_\delta -G_\delta+g_\delta }\omega_Y^n \\
&=  e^{Av_\delta + A\delta u_\delta - A\delta\rho + AC_1\delta + n\log(1-\delta) - G_\delta + g_\delta} \omega_Y^n\\
&\geq   e^{Av_\delta  - G_\delta + g_\delta} \omega_Y^n,
\end{aligned}
\]
where we used the bound $\|u_\delta \|_{L^\infty}+nA^{-1}\leq  C_1 $ and $\rho\leq 0$ in the last inequality. 

Now the function $v_{ \delta, \epsilon}:= \max(v_\delta, u_\epsilon)$ is $(\pi^*\omega_X+ \epsilon\omega_Y)$-psh and continuous on $Y$ since $v_\delta=-\infty$ on $E$. By \cite[Lemma 1.10]{GP}, we have
\[
( \pi^*\omega_X+ \epsilon \omega_Y +\ddbar v_{\delta, \epsilon})^n\geq e^{A v_{\delta,\epsilon}  }\min(e^{-G_\delta+g_\delta}, e^{-G_\epsilon+g_\epsilon}) \omega_Y^n.
\]
Applying   \cite[Theorem 3.3]{GL} again we get 
$$v_{\delta,\epsilon} \leq u_\epsilon + C\Vert \min(e^{-G_\delta+g_\delta}, e^{-G_\epsilon+g_\epsilon}) - e^{-G_\epsilon+g_\epsilon}  \Vert^{1/n}_{L^p(\omega_Y^n)}\leq u_\epsilon + C\Vert e^{-G_\delta+g_\delta}- e^{-G_\epsilon+g_\epsilon}  \Vert^{1/n}_{L^p(\omega_Y^n)}  $$
hence 
\begin{equation*}
u_\delta\leq u_\epsilon  -\delta\rho +2\delta C_1 + C\Vert   e^{-G_\delta+g_\delta} - e^{-G_\epsilon+g_\epsilon}  \Vert_{L^p(\omega_Y^n)}^{1/n}
\end{equation*}
since $v_\delta\leq v_{\delta,\epsilon}$ and $\|u_\delta\|_{L^\infty}\leq C_1$.  Combining this with \eqref{eq_est_1} it follows that $u_\epsilon$ converges locally uniformly on $Y\setminus E$ to $u_\infty$, and so $u_\infty$ solves the equation 
  \[ (\pi^*\omega_X + \ddbar u_{\infty})^n = e^{Au_{\infty} - G +  g}\omega_Y^n.\]
Since $u$ also solves the same equation, by uniqueness we have $u_\infty = u$ and so $u_\epsilon$ converges locally uniformly in $Y\setminus E$ to $u$. Now we can use the assumption that $u$ itself is
  smooth on $\pi^{-1}(X^{\text{reg}}\setminus D)$, together with Savin's
  small perturbation result~\cite{Savin}, to deduce that on any compact set away
  from $\pi^{-1}(X^{\text{reg}}\setminus D)$ we have smooth convergence
  $u_\epsilon \to u$.

  Now we will estimate the Ricci curvature of $\omega_{\epsilon}$. We have
  \[ \begin{aligned} \mathrm{Ric}(\omega_\epsilon) &= -A\ddbar u_\epsilon+ \ddbar
      G_\epsilon- \ddbar g_\epsilon+ \mathrm{Ric}(\omega_Y) \\
      &= -A\omega_\epsilon + A\omega_X + A\epsilon\omega_Y + \ddbar G_\epsilon -
      \ddbar g_\epsilon +\mathrm{Ric}(\omega_Y)  \\
      &\geq -A\omega_\epsilon-2A\omega_X + \alpha_\epsilon - \beta_\epsilon \end{aligned} \]
  on $Y$, where
  \[ \begin{aligned}
      \alpha_\epsilon &= \ddbar G_\epsilon + \mathrm{Ric}(\omega_Y)+\sum_i a_i\Theta_{h_i}+3A\omega_X \\
      \beta_\epsilon&=\ddbar g_{\epsilon}+\sum_i a_i\Theta_{h_i} , 
    \end{aligned}\]
    so that $  \beta_\epsilon = \ddbar(g_\epsilon - g)$ on $Y\setminus E$.

  From the estimates \eqref{eq:ddbG10}, \eqref{eq:Gpsh} and the local smooth convergence of $G_\epsilon$ to $G$ away from $\pi^{-1}(D)\cup E$, we have the following. For any $\kappa > 0$, we can choose a sufficiently small $\epsilon>0$, such if we denote by $N_\kappa$ the
  $\kappa$-neighborhood of $\pi^{-1}(D)\cup E$, then we have
  \begin{equation} \label{eq:Gklower}
    \alpha_\epsilon = \ddbar(G_\epsilon) +  \mathrm{Ric}(\omega_Y)+\sum_i a_i\Theta_{h_i} + 3A\omega_X
    \geq \begin{cases} 0, & \text{ on } Y\setminus
      N_\kappa, \\
      -C_2\omega_Y, &\text{ on }N_\kappa. \end{cases}. 
  \end{equation}
  
  Similarly, we have $\ddbar g_{\epsilon}\geq -C\omega_Y$, and at the same time we have smooth convergence $g_\epsilon \to g$ on compact sets away from $E$. Therefore $\beta_\epsilon$ satisfies a similar estimate to $\alpha_\epsilon$ above: for any $\kappa > 0$ we can choose $\epsilon > 0$ small enough so that 
  \[ 
  \beta_\epsilon 
    \geq \begin{cases} -\kappa\omega_Y, & \text{ on } Y\setminus
      N_\kappa, \\
      -C_2\omega_Y, &\text{ on }N_\kappa. \end{cases}. 
  \]

Let us supress the index $\epsilon$ for now, and write $\alpha=\alpha_\epsilon, \beta=\beta_\epsilon$. 
Write $|\alpha_-|, |\beta_+|$ for the norms of the negative and positive parts of $\alpha, \beta$, with respect to $\omega_\epsilon$. We have $|\alpha_-|= 0$ outside of $N_\kappa(D'\cup E)$, and $|\alpha_-| \leq C\mathrm{tr}_{\omega_\epsilon}\omega_Y$ on all of $Y$. This implies that 
\[ \int_Y |\alpha_-|\, \omega_\epsilon^n \leq C n \int_{N_{\kappa}(D'\cup E)} \omega_Y\wedge \omega_\epsilon^{n-1}. \]
which goes to $0$ as $\kappa\to 0$ because of the uniform $L^\infty$ bound for the potentials of $\omega_\epsilon$. 

On the other hand we have $\beta \geq -\kappa\omega_Y$ outside of $N_\kappa(D'\cup E)$, and $\beta\geq -C\omega_Y$ inside $N_\kappa(D'\cup E)$, hence
\[ \begin{aligned} 
\int_Y |\beta_+|\, \omega_\epsilon^n & \leq \int_{Y\setminus N_\kappa}|\beta + \kappa\omega_Y|\, \omega_\epsilon^n +  \int_{N_\kappa}|\beta + C\omega_Y |\,\omega_{\epsilon}^n \\
&= n\int_{Y\setminus N_\kappa} (\beta + \kappa\omega_Y)\wedge \omega_{\epsilon}^{n-1} + 
n\int_{N_\kappa} (\beta + C\omega_Y)\wedge \omega_\epsilon^{n-1} \\
&\leq n\int_Y \beta\wedge \omega_\epsilon^{n-1} 
+ n\kappa\int_Y \omega_Y\wedge \omega_\epsilon^{n-1} + Cn\int_{N_\kappa} \omega_Y\wedge \omega_\epsilon^{n-1} \\
&= n \sum_{i} a_i [E_i]\cdot [\pi^{\star}\omega_X+\epsilon\omega_Y]^{n-1} +  n\kappa\int_Y \omega_Y\wedge \omega_\epsilon^{n-1} + Cn\int_{N_\kappa} \omega_Y\wedge \omega_\epsilon^{n-1}
\end{aligned}\]
where $E = \sum_i a_i E_i$ are exceptional divisors, as in \eqref{eq:gdefn}. 

If we now let $\epsilon\to 0$, so that also $\kappa\to 0$, then we see that
\[ \lim_{\epsilon\to 0} \int_Y (|\alpha_-| + |\beta_+|) \omega_\epsilon^n = 0. \]
Since we have $\mathrm{Ric}(\omega_\epsilon) + 2(\omega_\epsilon+\omega_X) \geq \alpha_\epsilon - \beta_\epsilon$, it follows that 
$\Vert (\mathrm{Ric}(\omega_\epsilon) + 2A(\omega_\epsilon + \omega_X))_-\Vert_{L^1}\to 0$.
This completes the proof of Proposition~\ref{prop:L1approx}.
\end{proof}

We can now follow the strategy of \cite[Theorem 16]{Sz24} to prove Theorem~\ref{thm:generalKcurrent}, which we restate here. 
\begin{theorem}\label{thm:Kcurrentproof} Suppose that $\omega = \omega_X + \ddbar u$ is a singular K\"ahler  metric on the normal K\"ahler variety $(X,\omega_X)$, satisfying the conditions (1), (2), (3) in Theorem~\ref{thm:generalKcurrent}. Then there exists a constant $c>0$ such that $\omega\geq c \omega_X$. 
\end{theorem}
\begin{proof}
  For the reader's convenience we sketch the proof. Let us consider the tame approximation of $\omega$ by metrics $\omega_\epsilon$ on the resolutions $\pi:Y\to X$ constructed in Proposition~\ref{prop:L1approx}. The Chern-Lu inequality, applied to the map $\pi: (Y, \omega_\epsilon) \to (X, \omega_X)$ implies that away from $D'\cup E$ we have
 \[ \Delta_{\omega_\epsilon} \log\mathrm{tr}_{\omega_\epsilon} \pi^\ast\omega_X \geq \frac{ g^{i\bar l} g^{k\bar j} \mathrm{Ric}(\omega_\epsilon)_{i\bar j} h_{k\bar l}}{ \mathrm{tr}_{\omega_\epsilon} \pi^\ast\omega_X} - B\mathrm{tr}_{\omega_\epsilon} \pi^\ast\omega_X. \]
  Here $g_{i\bar j}$ and $h_{i\bar j}$ denote the metric components of $\omega_\epsilon$ and $\pi^*\omega_X$ respectively, and $B > 0$ is a constant independent of $\epsilon$. Recall that $\omega_\epsilon = \pi^*\omega_X + \epsilon \omega_Y + i\partial\bar\partial u_\epsilon$, where $u_\epsilon$ is bounded in $L^\infty$ independently of $\epsilon$. 
  It follows that we have 
  \[ \Delta_{\omega_\epsilon} (-u_\epsilon) \geq \mathrm{tr}_{\omega_\epsilon} \pi^\ast\omega_X - n. \]
  For a suitable constants $C_1, C_2 >0$ we therefore have
  \[ \Delta_{\omega_\epsilon} (\log\mathrm{tr}_{\omega_\epsilon} \pi^*\omega_X - C_1u_\epsilon) \geq - |(\mathrm{Ric}(\omega_\epsilon) + 2A(\omega_\epsilon + \omega_X))_-| - C_2, \]
  on $Y\setminus (D'\cup E)$. We define the function $\Phi = \max\{0, \log\mathrm{tr}_{\omega_\epsilon} \pi^*\omega_X - C_1u_\epsilon\}$, which also satisfies
  \[ \Delta_{\omega_\epsilon} \Phi \geq - |(\mathrm{Ric}(\omega_\epsilon) + 2A(\omega_\epsilon+ \omega_X))_-| - C_2, \]
  in a distributional sense. Denoting by $H(x,y,t)$ the heat kernel on $(Y,\omega_\epsilon)$, it follows that
  \[ \partial_t \int_Y \Phi(y) H(x,y,t)\, dy \geq \int_Y (- |(\mathrm{Ric}(\omega_\epsilon) + 2A(\omega_\epsilon+ \omega_X))_-|(y) - C_2) H(x,y,t)\, dy, \]
  where the integrals are with respect to $\omega_\epsilon^n$. 
  The fact that $\Vert (\mathrm{Ric}(\omega_\epsilon) + 2A(\omega_\epsilon+\omega_X))_-\Vert_{L^1} \to 0$, together with the Heat kernel bound from Theorem~\ref{thm:GPSS}, implies that for any $t_0>0$ once $\epsilon$ is sufficiently small we have
  \[ \partial_t \int_Y \Phi(y) H(x,y,t)\, dy \geq -2C_2, \]
  and so in particular
  \[ \int_Y \Phi(y) H(x,y,t_0)\, dy \leq 2C_2 + \int_Y \Phi(y) H(x,y,1)\, dy. \]
  Using the uniform bounds for $u_\epsilon$, it follows from this that 
  \[ \int_Y \Phi(y) H(x,y,t_0)\, dy \leq C_3, \]
  for a suitable constant $C_3 > 0$, which is independent of $t_0, \epsilon$. We can then first let $\epsilon\to 0$, and then $t_0\to 0$, to get the required estimate $\mathrm{tr}_\omega \omega_X \leq C_4$, which implies the claimed lower bound for $\omega$. 
\end{proof}

\section{Tame approximation for Ricci solitons}\label{sec:KRS}

Let $(X,D)$ be a log-Fano klt pair, and suppose that $T \cong (\mathbb{S}^1)^r$ is a real torus of dimension $r$ whose complexification $T_{\mathbb{C}}\cong (\mathbb{C}^{\ast})^r$ acts effectively and holomorphically on $X$, and that the action of $T_{\mathbb{C}}$ preserves $D$. Because $X$ is normal, we can identify holomorphic sections of sufficiently divisible powers of $-(K_X+D)$ with their restrictions to $X^{\text{reg}}$, where they can be viewed as plurianticanonical sections with zeros of certain orders along the support of $D$. In particular, there is a natural lift of the $T_{\mathbb{C}}$ action on $X$ to an action on the $\mathbb{Q}$-line bundle $K_X+D$ (which we interpret as an action on sufficiently divisible powers of $K_X+D$). Let $\phi$ be a smooth Hermitian metric on $-(K_X+D)$ whose curvature form is a K\"ahler metric $\omega_0$ on $X$. Recall that $\mu_{\phi}$ denotes the corresponding adapted measure. 

Given an $\omega_0$-plurisubharmonic function $u\in L^{\infty}(X)$ such that $e^{-u}\mu_{\phi}$ is $T$-invariant, its curvature $\omega_u := \omega_0 +i\partial \overline{\partial}u$ is a $T$-invariant closed positive current. The Hermitian metric $e^{-u}\mu_{\phi}$ produces a moment map $\mathbf{m}_u =(\theta_1(u),...,\theta_r(u)): X \to \mathbb{R}^r$ for the $T$-action, given by
$$\theta_j(u)(x):=\left. \frac{d}{dt} \right|_{t=0} \exp(ite_j)^{\ast} \log(e^{-u}\mu_{\phi}),$$
where $(e_1,...,e_r)$ is a choice of basis of the Lie algebra $\text{Lie}(T)$ of $T$. Moreover, the image $P$ of $\mathbf{m}_u$ is a polytope which only depends on $(X,D)$. Given $\xi \in \text{Lie}(T)$, with corresponding holomorphic vector field $V$ on $X$, we define $g_V: P\to \mathbb{R}$ by $g_V(\alpha):=e^{-\alpha(\xi)}$ for $\alpha \in P$. 

We then say that $(X,\omega_u,V)$ is a singular K\"ahler-Ricci soliton on $(X,D,T)$ if it satisfies the complex Monge-Amp\`ere equation
\begin{equation} \label{eq:solitoneq} g_{V}(\mathbf{m}_u)(\omega_0 + i \partial \overline{\partial} u)^n = e^{-u}\mu_{\phi},\end{equation}
where the left-hand side is a Radon measure on $X$ which does not charge pluripolar sets of $X$, defined in \cite{BN}. By \cite[Proposition 3.7]{HL}, $\omega_u$ is smooth on $X^{\text{reg}}\setminus D$, where it satisfies the K\"ahler-Ricci soliton equation
$$\mathrm{Ric}(\omega_u) + \mathcal{L}_V \omega_u = \omega_u.$$

The following is the soliton analogue of Proposition \ref{prop:L1approx}.

\begin{proposition} \label{solitontameapprox}
Suppose $(X,\omega,V)$ is a singular K\"ahler-Ricci soliton metric on $(X,D)$.
Then there is a tame approximation $(\omega_{\delta})$ on a resolution
$\pi:Y \to X$ and a holomorphic vector field $\widetilde{V}$ on $Y$ satisfying $\pi_{\ast}\widetilde{V}=V$ on $X^{\text{reg}}$, such that the following hold: 
\begin{itemize}
\item[(a)] $\lim_{\delta\to0}||(\mathrm{Ric}(\omega_{\delta})+\mathcal{L}_{\widetilde{V}}\omega_{\delta}+\omega_{\delta})_{-}||_{L^{1}(Y,\omega_{\delta})}=0,$
\item[(b)] The drift heat kernels $H_{\delta}$ of $(Y,\omega_{\delta},\log g_V(\mathbf{m}_{\delta}))$ converge smoothly on $\pi^{-1}(X^{\text{reg}})$ to the drift heat kernel $H$ of $(X,\omega,\log g_V(\mathbf{m}_{u}))$, and satisfy the uniform off-diagonal upper bound
$$\sup_{\delta} \sup_{x,y\in Y} H_{\delta}(x,y,t) \leq Ct^{-m}$$
for some $m\in (0,\infty)$, where $\mathbf{m}_{\delta}:Y\to \text{Lie}(T)$ is the moment map corresponding to the metric $\omega_{\delta}$.
\end{itemize}
\end{proposition}

\begin{proof} (a) 
Choose a smooth Hermitian metric $h_{0}$ on $-(K_{X}+D)$ whose curvature $\omega_{X}:=\Theta_{h_{0}}$
is a K\"ahler metric on $X$. Let $\mu_0$ be the corresponding adapted
volume measure. By averaging over $T$, we may assume $h_0$ is $T$-invariant. Then $\omega=\omega_{X}+i\partial\overline{\partial}u_{KRS}$
for some bounded $T$-invariant $\omega_X$-plurisubharmonic function $u_{KRS}$,
and the corresponding weighted Monge-Ampere measure $\text{MA}_{g_{V}}(u_{KRS})$ constructed in \cite{BN} satisfies
\[
\text{MA}_{g_{V}}(u_{KRS})=ce^{-u_{KRS}}\mu_{0}
\]
for some $c>0$. Let $\pi:Y\to X$
be a $T$-equivariant resolution of singularities which is an isomorphism over $X^{\text{reg}}$, which exists by \cite[Proposition 3.9.1, Theorem 3.36]{Kollar}. Set $\widetilde{V}_{0}:=(\pi|_{X^{\text{reg}}})_{\ast}^{-1}V\in H^{0}(Y_{0},T^{1,0}Y),$ where $Y_{0}:=\pi^{-1}(X^{\text{reg}}\setminus D)$. Then $\widetilde{V}_{0}$ has
a unique extension to some $\widetilde{V} \in H^{0}(Y,T^{1,0}Y)$, and there is a $T_{\mathbb{C}}$-action on $L:=-\pi^{\ast}(K_{X}+D)$ induced by the natural action of $T_{\mathbb{C}}$ on $-(K_X+D)$. 

Let $\mathbf{m}_{0,u}$ be the moment map associated to the $T$-action on $Y$ and the $T$-invariant semipositive current $\pi^{\ast}\omega_X+i\partial \overline{\partial}u$. The singular K\"ahler-Ricci soliton equation on $(X,V,T)$ is then equivalent to the equation
\[
g_{V}(\mathbf{m}_{0,u})(\pi^{\ast}\omega_{X}+i\partial\overline{\partial}u)^{n}=e^{u+F}\Omega_{Y}
\]
on $Y$, where $u:=\pi^{\ast}u_{KRS}$, $g_V(\mathbf{m}_{0,u})=\pi^{\ast}g_V(\mathbf{m}_{u_{KRS}})$, $\Omega_{Y}\in\mathcal{A}^{n,n}(Y)$
is a smooth volume form satisfying 
\[
\mathrm{Ric}(\Omega_{Y})=\pi^{\ast}\omega_{X}-\sum_i a_{i}\Theta_{h_i} + \sum_i d_i \Theta_{h_{D_i'}},
\]
where $a_i,E_i,s_i,h_i$ are as in Proposition \ref{prop:L1approx}, $\sum_i d_i D_i'$ is the proper transform of $D$ with respect to $\pi$, and where 
\[
F=-2u+\sum_i a_{i}\log|s_i|_{h_i}^{2} - \sum_i d_i \log |s_{D_i'}|_{h_{D_i'}}^2.
\]

Because $E_i,D_i'$ are $T_{\mathbb{C}}$-invariant, there is a natural linearization of the $T_{\mathbb{C}}$ action; we can thus average over $T$ to assume that $|s_i|_{h_i}$, $|s_{D_i'}|_{h_{D_i'}},\Omega_Y$ are $T$-invariant.
It was shown in \cite{BN} that $g_{V}(\mathbf{m}_{0,u})(\pi^{\ast}\omega_{X}+i\partial\overline{\partial}u)^{n}$
is a well-defined measure on $Y$ which does not charge pluripolar
sets, so pushes forward under $\pi$ to the measure $\text{MA}_{g_{V}}(u_{KRS})$
on $X$. Choose any $T$-invariant K\"ahler metric $\omega_Y$ on $Y$. For each $\epsilon>0$, set $\eta_{\epsilon}:=\pi^{\ast}\omega_{X}+\epsilon\omega_{Y}$. For
$j\in\mathbb{N}$, we let $\phi_{j}\in C^{\infty}(Y)\cap PSH(Y,\pi^{\ast}\omega_{X}+j^{-1}\omega_{Y})$
be a uniformly bounded sequence of $T$-invariant functions which converge locally
smoothly on $Y_{0}$ to $u$. 

For any $T$-invariant $\eta_{\epsilon}$-plurisubharmonic function $u$ on $Y$, we let $\mathbf{m}_{\epsilon,u}$ be the moment map corresponding to $\eta_{\epsilon}+i\partial \overline{\partial}u$. Consider the Monge-Amp\`ere equations 
\begin{equation}
g_V(\mathbf{m}_{\epsilon,u_{\delta}})(\eta_{\epsilon}+i\partial\overline{\partial}u_{\delta})^{n}=e^{u_{\delta}+F_{\delta}}\Omega_{Y},\label{SolitonMA}
\end{equation}
where for some $j>2/\epsilon$, we set
\[
F_{\delta}:=-2\phi_{j}-\sum_i d_i \log(|s_{D_i'}|_{h_{D_i'}}^2+\delta)+\sum_i a_i\log(|s_i|_{h_i}^{2}+\delta).
\]

The equation (\ref{SolitonMA}) has unique $T$-invariant smooth solutions $u_{\delta}$  by a standard application of the continuity method: the $C^0$ estimate is automatic, $|\log g_V(\mathfrak{m}_{\epsilon,u_{\delta}})|$ is bounded because the polytopes corresponding to $\eta_{\epsilon}$ are bounded, and the higher order estimates are as in \cite[Section 6]{Zhu00} or \cite[Proposition 3.7]{HL}. In particular, $\omega_{\delta}=\eta_{\epsilon}+i\partial\overline{\partial}u_{\delta}$ are smooth $T$-invariant K\"ahler metrics. We set $f_{\delta}:=- \log g_V(\mathbf{m}_{\epsilon,u_{\delta}})$, so that $i\partial \overline{\partial}f_{\delta}=\mathcal{L}_{\widetilde{V}}\omega_{\delta}$, and we can argue as in Proposition \ref{prop:L1approx} to estimate
\begin{align*}
\mathrm{Ric}(\omega_{\delta})
\geq  -\omega_{\delta}-\mathcal{L}_{\widetilde{V}}\omega_{\delta}+\alpha-\beta,
\end{align*}
where 
\[\alpha = \sum_i d_i (\ddbar\log(|s_{D_i'}|_{h_{D_i'}}^2+\delta)+\Theta_{D'_i})+\sum_{a_i<0} (-a_i)(\ddbar\log(|s_i|_{h_i}^{2}+\delta)+\Theta_{E_i})\]
\[\beta = \sum_{a_i>0} a_i(\ddbar\log(|s_i|_{h_i}^{2}+\delta)+\Theta_{E_i}).\]
Because the polytopes corresponding to $\eta_{\epsilon}$ are bounded uniformly in $\epsilon$, we know that $|f_{\delta}|\leq C$ independently of $\epsilon,\delta,j$. Thus we have the uniform $C^{0}$ estimate $\sup_{Y}|u_{\delta}|\leq C$. The $C^{2}$ and higher regularity estimates are proved as in \cite[Proposition 3.7]{HL}. Therefore we infer that $(\omega_\delta)$ is a tame approximation for $\omega$. The remainder of the proof of (a) follows exactly as in Proposition \ref{prop:L1approx}.

(b) By \cite[Theorem 4.1]{GPSS2}, there exist $B<\infty$ and $\gamma>1$ such that such that for all $\delta>0$ and $v\in C^1(Y)$, we have
$$\left( \int_Y |v|^{2\gamma} \omega_{\delta}^n \right)^{\frac{1}{\gamma}} \leq B\int_Y (|\nabla v|_{\omega_{\delta}}^2+|v|^2) \omega_{\delta}^n.$$
Because 
$$|f_{\delta}|\leq  \sup_{\zeta \in P_{\delta}} |\langle \zeta ,\xi\rangle| <\infty$$
for each $\delta \in (0,1]$, where $P_{\delta}$ is a uniformly bounded sequence of polytopes in $\text{Lie}(T)$, we obtain $B'<\infty$ such that for all $\delta \in (0,1]$ and $v\in C^1(Y)$, we have
$$\left( \int_Y |v|^{2\gamma} e^{-f_{\delta}}\omega_{\delta}^n \right)^{\frac{1}{\gamma}} \leq B'\int_Y (|\nabla v|_{\omega_{\delta}}^2 + |v|^2) e^{-f_{\delta}} \omega_{\delta}^n.$$
That is, the weighted manifolds $(Y,\omega_{\delta},f_{\delta})$ satisfy a uniform Sobolev inequality. The proof \cite[Section 5.1]{GPSS2} of \cite[Formula 2.6]{GPSS2} can be modified by replacing the heat kernel with the drift heat kernel $K$. By using $L^p$ norms with respect to the weight $e^{f_{\delta}}$, one obtains
$H_{\delta}:Y\times Y\times (0,\infty)\to (0,\infty)$ satisfying Gaussian estimates of the form
$$H_{\delta}(x,y,t)\leq \frac{C}{t^{C}}\exp \left( -\frac{d_{\omega_{\delta}}^2(x,y)}{Ct} \right)$$
for some $C<\infty$ independent of $\delta>0$. In particular, the claimed estimate holds with $m=C$. Moreover, after passing to a subsequence, the smooth convergence $\omega_{\delta}\to \pi^{\ast}\omega_{KRS}$ on $Y\setminus (D'\cup E)$ implies that as $\delta \searrow 0$, the functions $H_{\delta}$ converge in $C_{loc}^{\infty}$ to some smooth function $H: (Y\setminus (D'\cup E))\times (Y\setminus (D'\cup E))\times (0,\infty)\to [0,\infty)$ satisfying the heat equation. Then the claim follows from the argument of \cite[Lemma 10.3]{GPSS2}.
\end{proof}

Given the previous approximation Lemma, the proof of the K\"ahler current property for $\omega$ proceeds similarly to Theorem~\ref{thm:Kcurrentproof} or \cite[Theorem 16]{Sz24}.  

\begin{proposition}\label{prop:KRS} Suppose $(X,\omega,V)$ is a singular K\"ahler-Ricci soliton on $(X,D)$. Then $\omega$ is a K\"ahler current. 
\end{proposition}
\begin{proof} Let $\omega_{\delta}$ be the tame approximation from Lemma \ref{solitontameapprox}, and let $\eta_{FS}$ denote the pullback of a normalized Fubini-Study metric by a projective embedding of $X$ induced by a power of $-(K_X+D)$. Because the action of $T_{\mathbb{C}}$ is the restriction of a linear torus action on projective space, we may average the Fubini-Study metric over $T$ in order to assume that $\eta_{FS}$ is $T$-invariant. In a local holomorphic chart, let $g_{i\overline{j}}$, $h_{i\overline{j}}$ be the components of the metric tensor of $\omega_\delta,\pi^{\ast}\eta_{FS}$,
respectively, and let $\xi^{k}$ be the components of a holomorphic
vector field $\xi$. Then we can compute
\begin{equation*}
\xi\log\text{tr}_{\omega_{\delta}}(\pi^{\ast}\eta_{FS})
= -\frac{g^{\overline{q}i}g^{\overline{j}p}h_{i\overline{j}}(\mathcal{L}_{\xi}g)_{p\overline{q}}}{\text{tr}_{\omega_{\delta}}(\pi^{\ast}\eta_{FS})}+\frac{\text{tr}_{\omega_{\delta}}(\mathcal{L}_{\xi}(\pi^{\ast}\eta_{FS}))}{\text{tr}_{\omega_{\delta}}(\pi^{\ast}\eta_{FS})}.
\end{equation*}
Take $\xi = \widetilde{V}$, which satisfies 
$$\text{Im}(\widetilde{V}) \log \text{tr}_{\omega_{\delta}}(\pi^{\ast}\eta_{FS}) = 0$$
since both $\omega_{\delta},\pi^{\ast}\eta_{FS}$ are $T$-invariant. Thus combining with the Chern-Lu inequality gives
\begin{align*} \Delta_{\omega_{\delta}}\log \text{tr}_{\omega_{\delta}}(\pi^{\ast}\eta_{FS}) &- \langle \nabla^{g_{\delta}} \log g(\mathbf{m}_{\delta}),\nabla^{g_{\delta}} \log \text{tr}_{g_{\delta}}(\pi^{\ast}\eta_{FS})\rangle_{\omega_{\delta}} = (\Delta_{\omega_{\delta}} -\widetilde{V})\log \text{tr}_{\omega_{\delta}}(\pi^{\ast}\eta_{FS}) \\& \geq \frac{g^{\overline{\ell}i}g^{\overline{j}k}(\operatorname{Ric}(g_{\delta})+ \mathcal{L}_{\widetilde{V}}g_{\delta} -\lambda g_{\delta})_{i\overline{j}}}{\text{tr}_{\omega_{\delta}}(\pi^*\eta_{FS})}-A-A\text{tr}_{\omega_{\delta}}(\pi^{\ast}\eta_{FS})
\end{align*}
for some $A<\infty$ independent of $\delta>0$.

Recalling that $\log g(\mathbf{m}_{\delta})=\langle \xi, \mathbf{m}_{\delta}\rangle$ for some fixed $\xi \in \mathfrak{g}$, we have
\begin{align*} \Delta_{\omega_{\delta}} u - \langle \nabla^{g_{\delta}} \log g(\mathbf{m}_{\delta}),\nabla^{g_{\delta}}u\rangle & =  n - \text{tr}_{\omega_{\delta}}(\pi^{\ast}\omega_X + \epsilon \omega_Y) - \langle \xi, \mathbf{m}_{\delta}-\mathbf{m}_{\pi^{\ast}\omega_X - \epsilon \omega_Y} \rangle \\ & \leq - C^{-1}\text{tr}_{\omega_{\delta}}(\pi^{\ast}\eta_{FS}) +C
\end{align*}
for some $C>0$, since the images of $ \mathbf{m}_{\delta} $ and   $\mathbf{m}_{\pi^*\omega_X+\epsilon \omega_Y}$  stay within a bounded set of $\mathfrak{t}^\vee$ for all $\epsilon$ small. 
We can now argue exactly as in the proof of Theorem~\ref{thm:Kcurrentproof}, using the drift heat kernel $H_\delta$ of $(Y, \omega_\delta, \tilde{V})$ instead of the heat kernel. 
\end{proof}

\section{The RCD property with $L^p$ control of Ricci}\label{sec:RCD}
In this section we prove Theorem~\ref{thm:LpRCD}. We do not actually need the metric on $(X, D)$ to be K\"ahler-Einstein, and have the following more general result. 
\begin{theorem}\label{thm:RCD2}
  Suppose that $\omega$ is a closed positive $(1,1)$-current with locally bounded potentials that is smooth on $X^{reg}\setminus D$ for a divisor $D$, and satisfies $\mathrm{Ric}(\omega) \geq \lambda\omega$ there. Define the metric measure space $\hat{X}$ as the metric completion of $X^{reg}\setminus D$ as above, equipped with the measure $\omega^n$. 
  Suppose that $\pi: Y \to X$ is a resolution, and $\omega$ has a tame approximation with metrics $\omega_\epsilon$ on $Y$ such that $\Vert \mathrm{Ric}_-(\omega_\epsilon)\Vert_{L^p(\omega_\epsilon)} < C$ for a constant $C$ independent of $\epsilon$, and $p > \frac{2n-1}{n}$. Then the metric measure space $\hat{X}$ is a non-collapsed $RCD(2n,\lambda)$-space. 
\end{theorem}
\begin{proof}
  By Proposition~\ref{prop:Honda} it is enough to show that eigenfunctions of the Laplacian on $(X^{\text{reg}}\setminus D, \omega)$ are Lipschitz. Suppose therefore that $u\in W^{1,2}$ satisfies $\Delta u = -bu$ on $X^{\text{reg}}\setminus D$, for a constant $b$. 
  
  We will first consider an eigenfunction $f$ of the Laplacian on $(Y, \omega_\epsilon)$ satisfying $\Delta f = -\mu f$, where $\mu < b+1$. Recall that we have the following improved Kato inequality for the functions $f$: for any $\kappa < \frac{1}{2n-1}$ (where $n$ is the complex dimension), there is a $C_\kappa > 0$, depending also on the dimension, such that
  \begin{equation}\label{eq:improvedKato}
  |\nabla^2 f|^2 \geq (1+\kappa)|\nabla|\nabla f||^2 - C_\kappa \mu^2 f^2. 
  \end{equation}
  To see this one can follow the calculation in Munteanu~\cite{M12} (see also Cheng-Yau~\cite{CY75}). At a point in orthonormal frame $e_1,\ldots, e_{2n}$, where $e_1 = |\nabla f|^{-1} \nabla f$, we have 
  \[ |\nabla^2 f|^2 \geq |f_{11}|^2 + 2\sum_{\alpha > 1} |f_{1\alpha}|^2 + \frac{1}{2n-1} \left| \Delta f - f_{11}\right|^2. \]
  The inequality \eqref{eq:improvedKato} follows by using that $|\nabla|\nabla f||^2 = \sum_i f_{1i}^2$ and $\Delta f = -\mu f$. 
  
  For $\beta\in (0,1)$ we have the following, for any smooth function $f$
  \[ \begin{aligned}
  \Delta\Big( 1 + |\nabla f|^2\Big)^\beta &= 4\beta(\beta-1) \Big(1 + |\nabla f|^2\Big)^{\beta-2}|\nabla f|^2 |\nabla |\nabla f||^2 + \beta \Big( 1+ |\nabla f|^2\Big)^{\beta-1} \Delta |\nabla f|^2 \\
  &= \beta\big(1 + |\nabla f|^2\big)^{\beta-1} \left( 2|\nabla^2f|^2  - \frac{ 4(1-\beta) |\nabla f|^2}{1 + |\nabla f|^2} |\nabla |\nabla f||^2 + 2\mathrm{Ric}(\nabla f, \nabla f)\right) \\
  &\qquad + 2\beta\big( 1 + |\nabla f|^2\big)^{\beta-1} \langle \nabla f, \nabla \Delta f\rangle,
  \end{aligned}, \]
  where we used  Bochner's formula in the second equality.

  Let us define $\widetilde{\mathrm{Ric}} = \mathrm{Ric} - \lambda g$, so that
  \[ \mathrm{Ric}(\nabla f, \nabla f) = \widetilde{\mathrm{Ric}}(\nabla f, \nabla f) + \lambda|\nabla f|^2. \]
  Suppose now that $\Delta f = -\mu f$, and $\beta$ is chosen so that $2(1-\beta) < 1+\kappa$, for the $\kappa$ in \eqref{eq:improvedKato}. Then we obtain
  \begin{equation}\label{eq:D1}
  \begin{aligned}
  \Delta \Big(1 + |\nabla f|^2\Big)^\beta &\geq -2\beta C_\kappa \mu^2 f^2 - 2\beta (|\widetilde{\mathrm{Ric}}_-| - \lambda) |\nabla f|^{2\beta} + 
  2\beta\big( 1 + |\nabla f|^2\big)^{\beta-1} \langle \nabla f, \nabla \Delta f\rangle \\
  &\geq -2\beta C_\kappa \mu^2 f^2 - C_\beta |\nabla f|^2-  |\widetilde{\mathrm{Ric}}_-|^{\frac{1}{1-\beta}} - 1\\
  &\qquad +2\beta\big( 1 + |\nabla f|^2\big)^{\beta-1} \langle \nabla f, \nabla \Delta f\rangle, 
  \end{aligned}
  \end{equation}
  where $C_\beta$ depends on $\beta$ (recall also that we are taking $|\lambda|\leq 1$), and $|\widetilde{\mathrm{Ric}}_-|$ means the absolute value of the most negative eigenvalue of $\widetilde{\mathrm{Ric}}(\omega_\epsilon)$. 
  
  Now let $f_t$ denote the solution of the heat equation on $\omega_\epsilon$ with initial condition given by $f$. Since $f$ is an eigenfunction, we have $f_t = e^{-t\mu} f$. We also choose a point $x_0\in X^{\text{reg}}\setminus D$, and identify it with its preimage $\pi^{-1}(x_0)\in Y$. Let $\rho_t(y) = H(x_0, y, t)$ denote the heat kernel centered at $x_0$ on $(Y, \omega_\epsilon)$. Similarly to \cite[Proposition 15]{Sz24} we have the following, using also \eqref{eq:D1}, where we omit the volume form $\omega_\epsilon^n$ from the calculation. 
  \begin{equation}\label{eq:dds} \begin{aligned}
  \partial_s \int_Y \Big(1 + |\nabla f_{t-s}|^2\Big)^\beta\, \rho_s &= 
  \int_Y \beta \Big(1 + |\nabla f_{t-s}|^2\Big)^{\beta-1} \left( - 2\langle \nabla f_{t-s}, \nabla\Delta f_{t-s}\rangle \rho_s\right) + \Big(1 + |\nabla f_{t-s}|^2\Big)^\beta\, \Delta\rho_s \\
  &\geq  -2\beta C_\kappa \mu^2 \int_Y f_{t-s}^2\, \rho_s - C_\beta \int_Y |\nabla f_{t-s}|^2\, \rho_s - \int_Y |\widetilde{\mathrm{Ric}}_-|^{\frac{1}{1-\beta}}\, \rho_s - 1. 
  \end{aligned} \end{equation}

We choose $\beta$ large enough such that $2(1-\beta) < 1+\frac{1}{2n-1}$, and at the same time $(1-\beta)^{-1} < p$. This is possible by our assumption that $p > \frac{2n-1}{n}$. Note that then as $\epsilon \to 0$, we will have $\Vert \widetilde{\mathrm{Ric}}_-\Vert_{L^{(1-\beta)^{-1}}} \to 0$ by using Hölder's inequality, since we have a uniform $L^p$ bound for $|\widetilde{\mathrm{Ric}}_-|$, while at the same time $\widetilde{\mathrm{Ric}}_- \to 0$ locally smoothly on $\pi^{-1}(X^{\text{reg}}\setminus D)$. 

The upper bound for the heat kernel in Theorem~\ref{thm:GPSS} then implies that if we fix any $s_0 > 0$, then for $\epsilon < \epsilon(s_0)$ we will have
\[ \int_Y |\widetilde{\mathrm{Ric}}_-|^{\frac{1}{1-\beta}}\, \rho_s < 1. \]
Let us also assume that $\Vert f\Vert_{L^2} \leq 1$, which implies that $|f_{t-s}| \leq C_b$ for a constant $C_b$ depending on the upper bound for the eigenvalue $\mu$.
From \eqref{eq:dds} we then get that for $s\in (s_0, 2)$
\[ \partial_s \int_Y \Big(1 + |\nabla f_{t-s}|^2\Big)^\beta\,\rho_s + C_\beta f_{t-s}^2\, \rho_s  \geq  - C_1 (b+1)^2, 
\]
where $C_1$ depends  on $\beta, \kappa$, but not on $\epsilon, s_0$. 
We set $t=1+s_0$, and integrate this inequality from $s=s_0$ to $s=1+s_0$. We find that
\[ \int_Y \Big(1 + |\nabla f_1|^2\Big)^\beta \rho_{s_0} \leq \int_Y \Big(1+|\nabla f_0|^2\Big)^\beta\, \rho_{1+s_0} \leq C_2(b+1)^2, \]
where $C_2$ is independent of $\epsilon, s_0$ (as long as $\epsilon < \epsilon(s_0)$). Using the upper bound for the heat kernel again, and using that we have a bound $\Vert \nabla f\Vert_{L^2} < C_b$ depending on the eigenvalue (see \cite{GPSS2}), we have
\[ \int_Y \Big(1 + |\nabla f_1|^2\Big)^\beta \rho_{s_0} \leq C_b, \]
for a constant $C_b$ depending on the upper bound $b+1$ for the eigenvalue. Since $f_1 = e^{-\mu}f$, this implies that 
\begin{equation} \label{eq:Dfest1}
\int_Y |\nabla f|^{2\beta}\, \rho_{s_0} \leq C_b,
\end{equation}
as long as $\epsilon < \epsilon(s_0)$.

Next, we assume that $f = \sum_{j=1}^N f_j$ is decomposed into eigenfunctions on $(Y, \omega_\epsilon)$, corresponding to eigenvalues that are at most $b+1$. From Theorem~\ref{thm:GPSS} we can arrange that $N < C_b$ for a constant depending on $b$, but independent of $\epsilon$. Suppose that $\Vert f\Vert_{L^2} \leq 1$. Then it follows that $\Vert f_j\Vert_{L^2}\leq 1$ for all $j$. At the same time we have
\[ |\nabla f|^{2\beta} \leq \left(\sum_{j=1}^N |\nabla f_j|\right)^{2\beta} \leq \sum_{j=1}^N |\nabla f_j|^{2\beta}, \]
noting that $\beta \leq \frac{1}{2}$. 
Applying \eqref{eq:Dfest1} we get
\begin{equation}\label{eq:Dfest2}
\int_Y |\nabla f|^{2\beta}\, \rho_{s_0} \leq NC_b \leq C_b^2,
\end{equation}
as long as $\epsilon < \epsilon(s_0)$. 

Let us now consider the eigenfunction $u$ on $(X^{\text{reg}}\setminus D, \omega)$ satisfying $\Delta u = -bu$. We can assume that $\Vert u\Vert_{L^2} \leq 1$, which implies that we have 
\[ \sup_X |u| + \int_X |\nabla u|^2\, \omega^n < C_b, \]
for a constant depending on $b$. We will apply the estimate \eqref{eq:Dfest2} to a suitable approximation of $u$ on $Y$. First, recall from \cite[Lemma 3.2]{Song14} that we have an exhaustion $K_1\subset K_2\subset\ldots$ of $X^{\text{reg}}\setminus D$ and functions $\phi_i$ that are supported in $X^{\text{reg}}\setminus D$, equal to 1 on $K_i$ and $\int |\nabla \phi_i|^2 \to 0$ as $i\to\infty$. We use these to define the functions $u_i = \phi_i u$. 

Now we decompose 
\[ u_i = \sum_j \Phi_j \]
into eigenfunctions corresponding to eigenvalues $\lambda_j$. By Lemma~\ref{lem:eigenfunction-convergence}, we know that on $(Y, \omega_\epsilon)$ there are corresponding eigenfunctions $\Phi_{j, \epsilon}$ converging to the $\Phi_j$ locally smoothly on $X^{\text{reg}}\setminus D$. Then we define the functions
\[ f_{i, \epsilon} = \sum_{\lambda_j < b+1} \Phi_{j,\epsilon}, \]
on $(Y, \omega_\epsilon)$. We can apply \eqref{eq:Dfest2} to the $f_{i,\epsilon}$, and then for fixed $s_0 > 0$ we can let $\epsilon\to 0$. The $f_{i,\epsilon}$ converge to $\tilde{u}_i$, where 
\[ \tilde{u}_i = \sum_{\lambda_j < b+1} \Phi_j, \]
and so we deduce that 
\[ \int_X |\nabla \tilde{u}_i|^{2\beta}\, \rho_{s_0} \leq C_b^2. \]
Letting $s_0 \to 0$ we find that $|\nabla \tilde{u}_i|(x_0) \leq C_b^2$. Finally, the $u_i$ converge locally smoothly on $X^{\text{reg}}\setminus D$ to $u$, so we get the required gradient estimate for $u$. 
\end{proof}

The following Lemma justifies the above claim of convergence of eigenfunctions. 
\begin{lemma}\label{lem:eigenfunction-convergence}
Let $0=\lambda_0<\lambda_1\leq \cdots\leq \lambda_k\leq \cdots$ be the sequence of eigenvalues of $\hat X$, and  $\lambda_{k, \epsilon}$ be the corresponding Laplace eigenvalues for $(Y, \omega_{\epsilon})$. Then after possibly passing to a subsequence in $\epsilon$, we have
\[\lim_{\epsilon\to 0}\lambda_{k, \epsilon} = \lambda_k\]
Moreover, the eigenfunctions converge in $C^{\infty}_{loc}(X^{\text{reg}})$.  (up to reparamerization of each eigenspace. )
\end{lemma}

\begin{proof}
Let $\{u_i\}_{i = 0}^\infty\subset L^{\infty}\cap H^1(\hat X)$ be an orthonormal sequence of eigenfunctions of $\hat X$ with $\Delta u_i = -\lambda_iu_i$. For any $k$, let $\tilde u_{k} = \phi_{i}u_k$ where $\phi_i$ are the cutoff functions from \cite[Lemma 8.2]{GPSS2}. Now fix $k$ such that $\lambda_{k+1}>\lambda_k$, and consider the $k+1$-dimensional subspace \[E = \text{Span}_{\mathbb R}\{\tilde u_{0}, \ldots,  \tilde u_{k}\}\]
then for any $i\gg 1$, there exist $\epsilon_0\ll 1$ such that for any $\epsilon<\epsilon_0$, we have
\[\int_{Y}\tilde u_j\tilde u_k \omega_{\epsilon}^n =  \delta_{jk}+o(1)\]
and 
\[\int_Y\langle \nabla\tilde u_j, \nabla\tilde u_k\rangle_{\omega_{\epsilon}}\omega_{\epsilon}^n = \delta_{jk}\lambda_j+o(1)\]
where the $o(1)$ term can be made arbitrarily small by taking $i$ large. Therefore for any $u\in E$, we have 
\[\frac{\int_Y|\nabla u|_{\omega_{\epsilon}}^2\omega_{\epsilon}^n}{\int_Y u^2\omega_{\epsilon}^n}\leq \lambda_k+o(1)\]
and it follows from the Rayleigh quotient characterization of eigenvalues that \[\limsup_{\epsilon \to 0}\lambda_{k, \epsilon}\leq \lambda_k.\] 

On the other hand, if we let $u_{0, \epsilon}, \ldots, u_{k, \epsilon}$ be the corresponding sequence of orthonormal eigenfunctions corresponding to $\lambda_{0, \epsilon}\ldots, \lambda_{k, \epsilon}$ on $(Y, \omega_{\epsilon})$, then by the Sobolev inequality and Moser's iteration method from \cite[Lemma 6.2]{GPSS2}, we have $\|u_{k, \epsilon}\|_{L^{\infty}}\leq C$ is uniformly bounded independent of $\epsilon$, and $\lambda_{k, \epsilon}$ is also uniformly bounded independent of $\epsilon$, therefore by passing to a subsequence and letting $\epsilon \to 0$, the tuple $(u_{0, \epsilon}, \ldots, u_{k, \epsilon})$ converges to $k+1$-orthonormal eigenfunctions of $\hat X$ with eigenvalues less than or equal to $\lambda_k$, and therefore the eigenvalues $\lambda_{k, \epsilon}$ (and the eigenfunctions) must converge to $\lambda_k$ (and and the eigenfunctions of $\hat X$) after passing to a subsequence. 
\end{proof}

As a consequence of Theorem \ref{thm:LpRCD}, we get the RCD property for certain K\"ahler-Einstein spaces $(X, \omega_{KE})$ with general automorphism groups.  This extends  \cite[Theorems 3,4]{Sz24} where it was assumed that the automorphism group is discrete, and approximations using constant scalar curvature metrics were used. 
\begin{corollary}\label{cor:extremal}
Let $(X, \omega_{KE})$  be a K\"ahler-Einstein space with log terminal singularities with ${\rm Ric}(\omega_{KE})=\lambda \omega_{KE}$. 
Suppose that $ \omega_{KE}$ is   $\mathbb{K}$-invariant, where $\mathbb{K} $ is a maximal compact subgroup of the reduced automorphism group ${\rm Aut_{red}}(X)$ and $\mathbb{K} $  contains a  maximal real  torus   $T\simeq  (S^1)^r$ of ${\rm Aut_{red}}(X)$. 
Assume that $X$ admits  a $T  $-equivariant   resolution of singularities of Fano type $\pi: Y \rightarrow X$ and $\pi$ is an isomorphism over $\pi^{-1}(X^{\text{reg}})$. Then the metric completion   $\overline{(X^{\text{reg}}, d_{KE})}$,  equipped with the measure $\omega^n_{KE}$ is  a non-collapsed $RCD(2n, \lambda)$-space  and it is homeomorphic to $X$ when $X$ is a normal projective variety.
\end{corollary}
Here we say a resolution  of singularities $\pi: Y\rightarrow X$ is of {\it Fano type} (cf. \cite{BJT24}) if $Y$ admits a finite measure of form $\widehat{\mu}_Y=e^{-f}\mu_Y$ with $f$ a quasi-psh function and $\mu_Y$ a smooth volume form on $Y$, such that ${\rm Ric}(\widehat{\mu}_Y):={\rm Ric}(\mu_Y)+ i\partial\bar \partial f$ is $\pi$-semipositive, i.e. ${\rm Ric}(\widehat{\mu}_Y) + C\pi^* \omega_X\geq 0$ for some $C>0$ and a K\"ahler metric $\omega_X$ on $X$. A resolution  of singularities for which $-K_Y$ is relatively nef over $X$ is of Fano type. Another class of examples is provided by singular K\"ahler-Einstein varieties with \emph{admissible} singularities, defined by Li-Tian-Wang~\cite{LTW}. In this case $\widehat{\mu}_Y$ is given by the pullback of the adapted volume form on $X$ for a resolution $\pi:Y\to X$ whose discrepancies are all non-positive.

\begin{proof}

Since   $\omega_{KE}\in [\omega_X]$ is a singular $\mathbb{K} $-invariant K\"ahler-Einstein metric on $X$,  by \cite[Theorem 2.19]{Li22} (see also \cite{BBEGZ19})  the  Mabuchi functional ${\bf M}_\omega$ is $T ^\mathbb{C}$-coercive, where 
 $T^\mathbb{C}$ is  the complexification of $T$.  Then it follows from  \cite{BJT24} and \cite{PT24} that there is a {\it tame} approximation of   $\omega_{KE}$ with extremal K\"ahler metrics    $\widehat{\omega}_\varepsilon\in [\omega_\varepsilon]$ where $\omega_{\varepsilon}= \pi^\ast \omega_X+ \varepsilon \omega_Y$, for some $T$-invariant K\"ahler metric $\omega_Y$ on $Y$.  Recall  that for each pair  $(T, [\omega_{\varepsilon}])$, we can find a unique  vector field  $\xi_\varepsilon^{\rm ext}\in \mathfrak{t}={\rm Lie}(T)$, called {\it the extremal vector field} such that if $\widehat{\omega}_\varepsilon\in [\omega_\varepsilon]$ is a  $T$-invariant extremal metric   then $  \xi_\varepsilon^{\rm ext}=J \hat \nabla S(\widehat{\omega}_\varepsilon) $ is a Killing vector field for $\widehat{\omega}_\varepsilon$.  Let  $m_\varepsilon:   Y\rightarrow \mathfrak{t}^\vee$ be the  moment map for $(T, \omega_{\varepsilon})$, i.e. a $T$-invariant  smooth map $m_\varepsilon$  such that  $d\langle m_\varepsilon
, \xi \rangle= -\omega_\varepsilon(\xi, \cdot), \forall \xi\in  \mathfrak{t} $,  uniquely defined by the  normalization    $\int_Y \langle m_\varepsilon, \xi\rangle\omega_\varepsilon^n=0$. 

\smallskip
Then in  our setting, the extremal metric $\widehat{\omega}_\varepsilon \in [\omega_\varepsilon]$ satisfies $S(\widehat{\omega}_\varepsilon ) =  \langle m_\varepsilon, \xi_\varepsilon^{\rm ext}\rangle + s_{ \varepsilon} $, where  $$s_\varepsilon:= \frac{c_1(Y)\cdot \{\widehat{\omega}_\varepsilon\}^{n-1}}{\{\widehat{\omega}_\varepsilon\}^n}\rightarrow \lambda. $$ 
 It suffices to check that $\int_Y |{\rm Ric}(\widehat{\omega}_\varepsilon)|^2\widehat{\omega}_\varepsilon^n$ is uniformly  bounded. Indeed, first we  remark that 
\[
\int_Y |{\rm Ric}(\widehat{\omega}_\varepsilon)|^2\widehat{\omega}_\varepsilon^n= \int_Y S(\widehat{\omega}_\varepsilon)^2 \widehat{\omega}_\varepsilon^n-  n(n-1)c^2_1(Y)\cdot \{\widehat{\omega}_\varepsilon\}^{n-2}.
\]
Since $\omega_\varepsilon = \pi^*\omega_X +\varepsilon \omega_Y $ converges smoothly to $\pi^*\omega_X$, one can imply that $\xi_\varepsilon^{\rm ext} $ converges to 0  in $ \mathfrak{t} $ (cf. \cite{BJT24}) and  $m_{\varepsilon}(Y)$ is contained in a compact set of $\mathfrak{t}^\vee$ for all $\varepsilon$ small enough. Therefore  $S(\widehat{\omega}_\varepsilon ) =\langle m_\varepsilon, \xi_\varepsilon^{\rm ext}\rangle + s_{ \varepsilon}  $ is uniformly bounded and converges to $\lambda$ as $\varepsilon\rightarrow 0$.  Hence  $\int_Y |{\rm Ric}(\widehat{\omega}_\varepsilon)|^2\widehat{\omega}_\varepsilon^n$ is uniformly bounded and the conclusion follows from  Theorem~\ref{thm:LpRCD}.  When $X$ is a normal projective variety,  the metric completion   $\overline{(X^{\text{reg}}, d_{KE})}$ is homeomorphic to $X$ by \cite[Theorem 17]{Sz24}. 
\end{proof}

\end{document}